\documentclass[10pt]{amsart}

\usepackage{amsmath,amssymb}
\usepackage{xcolor}
\definecolor{wine}{rgb}{0.45,0.00,0.10}
\usepackage{hyperref}
\hypersetup{colorlinks=true, linkcolor=blue, citecolor=magenta, urlcolor=wine}
\usepackage{url}
\usepackage{fancyhdr}

\newtheorem{theorem}{Theorem}[section]
\newtheorem{prop}[theorem]{Proposition}
\newtheorem{lemma}[theorem]{Lemma}

\theoremstyle{definition}

\numberwithin{equation}{section}

\newcommand{\nn}{\mathbb{N}}
\newcommand{\qq}{\mathbb{Q}}
\newcommand{\rr}{\mathbb{R}}
\newcommand{\zz}{\mathbb{Z}}
\newcommand{\mca}{\mathcal{A}}
\newcommand{\uu}{\mathcal{U}}
\newcommand{\gp}{\text{gp}}
\newcommand{\supp}{\text{supp}}
\newcommand{\zplus}{\mathsf{Z}_+}
\newcommand{\lplus}{\ell_+}
\newcommand{\mset}[1]{\{\!\{#1\}\!\}}

\begin{document}

\mbox{}
\title{The Bi-UF Positive Conjecture Holds}
\author{Omar Graia}
\date{\today}

\begin{abstract}
	A subsemiring $S$ of the nonnegative cone of the real line containing $1$ is called a positive semiring. The Bi-UF Positive Conjecture, posed by Baeth, Chapman, and Gotti in 2021, states that the prototypical semiring $\nn_0$ is the only positive semiring whose additive monoid $(S,+)$ and nonzero multiplicative monoid $(S^\bullet,\cdot)$, where $S^\bullet=S\setminus\{0\}$, are both factorial. In this paper, we prove their conjecture.
\end{abstract}
\bigskip

\maketitle

\bigskip
\section{Introduction} \label{s:introduction}

\smallskip
Factorization theory for commutative semirings has two natural and interacting components: the factorization of elements under addition and their factorization under multiplication. Let $S$ be a semidomain, which is a subsemiring of an ambient field. We say that $S$ is bi-atomic if both its additive monoid $(S,+)$ and multiplicative monoid $S^\bullet=S\setminus\{0\}$ are atomic, which means that every element of the monoid can be obtained as a factorization of finitely many atoms (i.e., irreducible elements) of the given monoid. In addition, we say that a semidomain $S$ is a bi-UFS if both $(S,+)$ and $S^\bullet$ are atomic and their elements can be factored into finitely many atoms in a unique manner (up to order and associates). The prototypical example of a bi-UFS is $\mathbb{N}_0$: indeed, its additive monoid is freely generated by $1$, while its multiplicative monoid $(\nn,\cdot)$ has the usual unique factorization into primes.

In their study of bi-atomic classes of positive semirings, Baeth, Chapman, and Gotti~\cite{BCG21} conjectured that the prototypical example $\mathbb{N}_0$ is the only positive semiring that is a bi-UFS~\cite[Conjecture~7.7]{BCG21}. This assertion, now known as the Bi-UF Positive Conjecture, is striking because the additive and multiplicative factorization structures of a semiring constrain one another through distributivity, even though either monoid can be complicated when considered in isolation. Recall that a monoid is half-factorial if it is atomic and any two factorizations of the same element have the same length. Accordingly, a semidomain is a bi-HFS if both its additive and multiplicative monoids are half-factorial. Along with the Bi-UF Positive Conjecture, Baeth, Chapman, and Gotti posed in~\cite{BCG21} the following related question: Is $\mathbb{N}_0$ the only positive semiring that is a bi-HFS?

Several recent results have established the conjecture under additional hypotheses. In the preprint~\cite{AACGWW26}, Abedi et al. proved it for finitely generated rational positive semirings using only information from their corresponding additive monoids. Deng, Gotti, and Zeng proved that for $q \in \qq_{>0}$, the rational monogenic semidomain $\nn_0[q]$ is multiplicatively factorial if and only if $q \in \nn \cup \nn^{-1}$, and that $\nn_0$ is the only rational monogenic semidomain with the bi-HF property~\cite[Theorem~5.2 and Corollary~5.4]{DGZ26}. Gotti et al. proved that no positive quadratic monogenic semidomain is a bi-UFS~\cite[Corollary~3.7]{GGHL26}. They also proved that if $\alpha$ is a positive non-rational algebraic number and the positive-constant primitive integer multiple $w_\alpha(x)$ of its minimal polynomial has composite constant term, then $\nn_0[\alpha]$ is not a bi-UFS~\cite[Corollary~3.9]{GGHL26}, and established that every semidomain whose additive monoid is a finite-rank free commutative monoid is isomorphic to a finite $\nn_0$-span of algebraic numbers~\cite[Theorem~4.2]{GGHL26}. They also proposed the broader complex version of the conjecture, according to which every bi-UFS subsemiring of $\mathbb{C}$ should be isomorphic to $\nn_0$. The most general prior progress towards the Bi-UF Positive Conjecture was established by Bilakanti et al.~\cite{BGKLLPYY26}. They proved that $\nn_0[\alpha]$ is a bi-UFS for a positive algebraic number $\alpha$ if and only if $\alpha \in \nn$~\cite[Theorem~4.8]{BGKLLPYY26}, and more generally that if $\alpha_1,\dots,\alpha_n$ are positive algebraic numbers and $\nn_0[\alpha_1,\dots,\alpha_n]$ is a bi-UFS, then $\nn_0[\alpha_1,\dots,\alpha_n]=\nn_0$~\cite[Theorem~5.7]{BGKLLPYY26}. Thus, prior to the present work, the conjecture was known for finitely generated algebraic positive semidomains, while the general case without finite-generation or algebraicity assumptions remained open.

The bi-HF property is a relaxation of the bi-UF property, which was also introduced in~\cite{BCG21}. The question of whether $\mathbb{N}_0$ is also the only positive semiring with the bi-HF property was posed by Baeth, Chapman, and Gotti in their study of bi-atomic positive semirings~\cite[Question~7.8(1)]{BCG21}. Recent related progress, however, reveals that the bi-HF property is not strong enough to guarantee that the only bi-HF positive semiring is $\mathbb{N}_0$. The first negative answer was given by Gonzalez, Polo, and Rodriguez~\cite{GPR24}, with an explicit construction of a positive semiring distinct from $\mathbb{N}_0$ whose additive and multiplicative monoids are both half-factorial. More recently, Gotti et al.~\cite{GGHL26} generalized this construction and provided a new one. For related work on half-factoriality, see Chapman and Coykendall~\cite{CC00} and the recent work of Boynton, Coykendall, Moles, and Morrow~\cite{BCMM26}.

There is a parallel length-factorial problem. A monoid is a length-factorial monoid (LFM) if any two distinct factorizations of the same element have different lengths, and a semiring is a bi-LFS if both its additive and multiplicative monoids are LFMs. Baeth, Chapman, and Gotti also asked whether $\mathbb{N}_0$ is the only positive bi-LFS~\cite[Question~7.8(2)]{BCG21}. This question appears to remain open. The problem was later restated by Bu, Vulakh, and Zhao~\cite[Question~5.5]{BVZ}, where it is related to Conjecture~5.6. For related work on length-factoriality, see also Geroldinger and Zhong~\cite{GZLF} and Chapman, Coykendall, Gotti, and Smith~\cite{CCGSLF}.

In this paper, we resolve the Bi-UF Positive Conjecture. In the proof we present, we harness a coordinate cone given by additive factoriality and the primality given by multiplicative factoriality to produce inside a hypothetical counterexample an irrational element that is simultaneously an additive and a multiplicative atom. Our proof isolates two local additive obstructions and then combines them with multiplicative unique factorization. In Section~\ref{s:prelim} we establish the notation and preliminary lemmas. Then, in Section~\ref{s:packets}, we develop the two obstructions. Finally, in Section~\ref{s:conjecture}, we complete the proof of the conjecture. Our main result proves, in full generality, the conjecture that motivated this line of work.

\bigskip
\section{General notation and preliminary lemmas}
\label{s:prelim}

\smallskip

Throughout this paper, the term \emph{monoid} means a cancellative, commutative semigroup with identity. For background on factorization theory in commutative monoids, see Geroldinger and Halter-Koch~\cite{GH06}. If $M$ is a monoid, then $\uu(M)$ denotes its group of units and $\mca(M)$ denotes its set of atoms. A monoid is a \emph{UFM} if each nonunit has exactly one factorization into atoms, up to order and associates, and it is an \emph{HFM} if all factorizations of any given element have the same length. A positive semiring $S$ is a \emph{bi-HFS} (resp., \emph{bi-UFS}) provided that both $(S,+)$ and $(S^\bullet,\cdot)$ are HFMs (resp., UFMs). For a positive semiring $S$, we let $\mca_+(S)$ and $\mca_\times(S)$ denote the sets of atoms of $(S,+)$ and $(S^\bullet,\cdot)$, respectively, while $\uu(S)$ denotes the set of multiplicative units of $S$. We also set $\nn_0 := \nn \cup \{0\}$ and identify $n \in \nn_0$ with $n \cdot 1 \in S$.

\smallskip

When $(S,+)$ is a UFM, it is convenient to denote by $\zplus(x)$ the unique additive factorization of $x \in S$ written as a finite multiset of atoms of $(S,+)$. In particular, $\zplus(0)$ is the empty multiset and $\lplus(x) := |\zplus(x)|$. If $A$ and $B$ are finite multisets, then $A \uplus B$ denotes their multiset union. Finally, if $A$ is a multiset of elements of $S$ and multiplication by $s \in S^\bullet$ sends every member of $A$ to an additive atom, then $sA$ denotes the multiset obtained from $A$ by multiplying each of its elements by $s$.

\smallskip

From Geroldinger and Zhong~\cite[Section~2.2]{GZ20}, a cancellative commutative monoid is a UFM precisely when its reduced monoid is a free commutative monoid on the classes of atoms. Thus atoms in a UFM are prime. We will also use the following refinement property.

\begin{lemma} \label{l:refinement}
	Let $M$ be a UFM. If $x,y,z,t \in M$ satisfy $xy=zt$, then there are $g,h,k,l \in M$ such that $x=gh$, $y=kl$, $z=gk$, and $t=hl$.
\end{lemma}

\begin{proof}
    Let $\overline{a}$ denote the image of $a\in M$ in the reduced monoid $M_{\text{red}}$. Since $M_{\text{red}}$ is free commutative, the equality $\overline{x}\,\overline{y}=\overline{z}\,\overline{t}$ can be refined coordinatewise.  Thus there are $\overline{g},\overline{h},\overline{k},\overline{l} \in M_{\mathrm{red}}$ such that $\overline{x}=\overline{g}\,\overline{h}$, $\overline{y}=\overline{k}\,\overline{l}$, $\overline{z}=\overline{g}\,\overline{k}$, and $\overline{t}=\overline{h}\,\overline{l}$. Now choose representatives $g_0,h_0,k_0,l_0 \in M$.  Then there are units $\alpha,\beta,\gamma,\delta \in \uu(M)$ such that $x=\alpha g_0h_0$, $y=\beta k_0l_0$, $z=\gamma g_0k_0$, and $t=\delta h_0l_0$. Since $xy=zt$, cancellation gives $\alpha\beta=\gamma\delta$. Finally, set $g=g_0$, $h=\alpha h_0$, $k=\gamma k_0$, and $l=\beta\gamma^{-1}l_0$. Then $x=gh$, $y=kl$, $z=gk$, and $t=hl$.
\end{proof}

\smallskip

Now suppose that $S$ is a positive semiring whose additive monoid is a UFM, and put $G=\gp(S,+)$. By the standard factorial-monoid fact recalled above, $(S,+)$ is the free commutative monoid on $\mca_+(S)$, the group $G$ is the free abelian group on $\mca_+(S)$, and the image of $S$ in $G$ is the nonnegative coordinate cone. By the standard ring-of-differences construction for cancellative semirings, multiplication on $S$ extends to a ring structure on $G$ via $(x-y)(u-v)=(xu+yv)-(xv+yu)$~\cite[Theorem~2.1.4]{Kuber15}. We use this coordinate setup throughout. For $x \in G$ and $a \in \mca_+(S)$, we write $\operatorname{coeff}_a(x)$ for the $a$-coordinate of $x$ in this basis. The natural map $G \to \rr$ is injective: if a finite integral combination $\sum_i n_i a_i$ of additive atoms maps to $0$ in $\rr$, then separating positive and negative coefficients gives an equality of two elements of $S$ with the same additive factorization, forcing every $n_i$ to be zero. Thus multiplication by any nonzero element of $S$ is injective on $G$.

\smallskip

\begin{lemma} \label{l:length}
	Let $S$ be a positive semiring whose additive monoid is a UFM. Then the following statements hold.
	\begin{enumerate}
		\item[(1)] $\lplus(xy) \ge \lplus(x)\lplus(y)$ for all $x,y \in S^\bullet$.
		\smallskip
		\item[(2)] If $x_1,\dots,x_n \in S^\bullet$ and $a_1,\dots,a_n \in \mca_+(S)$ satisfy $x_1+\cdots+x_n=a_1+\cdots+a_n$, then $x_1,\dots,x_n$ are additive atoms and $\mset{x_1,\dots,x_n}=\mset{a_1,\dots,a_n}$.
	\end{enumerate}
\end{lemma}

\begin{proof}
	For (1), write $x=a_1+\cdots+a_m$ and $y=b_1+\cdots+b_n$ as sums of additive atoms. Then $xy=\sum_{i,j} a_i b_j$. Each product $a_i b_j$ is nonzero and so has additive length at least $1$, so $\lplus(xy) \ge mn$.

	For (2), refine each $x_i$ into additive atoms. This produces at least $n$ atoms, while the right side is an additive factorization of length $n$. By the additive-coordinate setup above, the refined left-side multiset equals the right-side multiset, and so the refined left side has exactly $n$ atoms. Hence each $x_i$ has length $1$, and the asserted multiset equality follows.
\end{proof}

\smallskip

\begin{lemma} \label{l:unit-action}
	Let $S$ be a positive semiring. Multiplication by a unit of $S$ permutes $\mca_+(S)$.
\end{lemma}

\begin{proof}
	Take $u \in \uu(S)$ and $a \in \mca_+(S)$. If $ua=x+y$ in $S$, then $a=u^{-1}x+u^{-1}y$, and so one of $x,y$ is zero. Hence $ua \in \mca_+(S)$. Applying the same argument to $u^{-1}$ shows that multiplication by $u$ is a permutation of $\mca_+(S)$.
\end{proof}

\smallskip

\begin{lemma} \label{l:divisor-closed}
	Let $S$ be a positive semiring whose additive monoid is a UFM. If $\alpha \in \mca_+(S)$ and $\alpha=de$ with $d,e \in S^\bullet$, then $d,e \in \mca_+(S)$.
\end{lemma}

\begin{proof}
	If $d$ were not an additive atom, then, since $(S,+)$ is a UFM and $d \in S^\bullet$, an additive factorization of $d$ would have length at least $2$. In particular, $d=s+t$ for some $s,t \in S^\bullet$. But then $\alpha=se+te$ would be a nontrivial additive decomposition of $\alpha$, a contradiction. Thus $d \in \mca_+(S)$, and the proof for $e$ is identical.
\end{proof}

\smallskip

\begin{lemma} \label{l:order-hole}
	Let $S$ be a positive semiring whose additive monoid is a UFM, and suppose that $1$ and $a$ are distinct additive atoms. For any positive integers $m,n$, no element of $S^\bullet$ can be equal, as a real number, to $ma-n$ or to $n-ma$.
\end{lemma}

\begin{proof}
	If $ma-n=s \in S^\bullet$, then $ma=n+s$. After factoring $s$ additively, this equality expresses $m$ copies of the atom $a$ as $n$ copies of the distinct atom $1$ together with at least one further atom, contradicting uniqueness of additive factorizations. The case $n-ma \in S^\bullet$ is identical.
\end{proof}

\smallskip

\begin{lemma} \label{l:scaling-chain}
	Let $b>1$ and $a \neq 1$ be positive real numbers, and let $M$ be a nonempty finite multiset of positive real numbers. If $\mset{1}\uplus bM=\mset{a}\uplus M$, then $a=b^N$ for some $N \ge 1$ and $M=\mset{1,b,\dots,b^{N-1}}$.
\end{lemma}

\begin{proof}
	Let $c(t)$ be the multiplicity of $t$ in $M$, with $c(t)=0$ outside $\supp(M)$. The multiset equality gives $\mathbf 1_{t=1}+c(t/b)=\mathbf 1_{t=a}+c(t)$ for every $t>0$. There are no elements of $M$ less than $1$, for a least such element $s$ would make the left side of this equality vanish and the right side positive. This also rules out $0<a<1$. Since $a \ne 1$, we have $a>1$. At $t=1$ the equality gives $c(1)=1$, and for $j \ge 1$ it gives $c(b^j)=c(b^{j-1})-\mathbf 1_{a=b^j}$. Since $M$ is finite, $a=b^N$ for some $N \ge 1$, and then $c(b^j)=1$ for $0 \le j < N$ and $c(b^j)=0$ for $j \ge N$. If $M$ contained an element outside $\{1,b,\dots,b^{N-1}\}$, then a least such element $s$ would satisfy $s \ne 1$ and $s \ne a$, so $c(s)=c(s/b)>0$. But $s/b<s$, and if $s/b$ were in the displayed chain, then $s$ would be in the chain or equal to $a$. This contradiction excludes all other elements of $M$.
\end{proof}

\smallskip

\begin{prop}[{\cite[Proposition~7.1(1)]{BCG21}}] \label{p:rational-collapse}
	If $S$ is a bi-HFS positive semiring, then $S \cap \qq = \nn_0$.
\end{prop}

\section{Two additive obstructions}
\label{s:packets}

\smallskip

\begin{theorem} \label{t:no-units}
	If $S$ is a bi-UFS positive semiring, then $\uu(S)=\{1\}$.
\end{theorem}

\begin{proof}
	Suppose for the sake of contradiction that $u \in \uu(S) \setminus \{1\}$. Since $(S,+)$ is atomic, \cite[Proposition~3.1]{BCG21} gives $\uu(S) \subseteq \mca_+(S)$. Hence every $u^n$, with $n \in \zz$, is an additive atom. These atoms are pairwise distinct: if $u^m=u^n$ with $m<n$, then $u^{n-m}=1$, and since $u>0$ this forces $u=1$. Set $P=1+u+u^2$, $Q=1+u^3$, $R=1+u$, and $T=1+u^2+u^4$. Since $PQ=RT$, Lemma~\ref{l:refinement} gives $g,h,k,l \in S^\bullet$ such that $P=gh$, $Q=kl$, $R=gk$, and $T=hl$. Now $\lplus(R)=2$, and Lemma~\ref{l:length}(1) yields $\lplus(g)\lplus(k) \le 2$. Since all lengths are positive, the possible pairs are $(1,2)$, $(2,1)$, and $(1,1)$. If $\lplus(g)=1$ and $\lplus(k)=2$, then writing $k=b+c$ with $b,c \in \mca_+(S)$ and applying Lemma~\ref{l:length}(2) to $gb+gc=1+u$ shows that one of $gb,gc$ is $1$, so $g$ is a unit. Similarly, the case $\lplus(g)=2$ and $\lplus(k)=1$ forces $k$ to be a unit. Hence either $g$ is a unit, or $k$ is a unit, or both $g$ and $k$ are nonunit additive atoms.

	Assume first that $g=e$ is a unit. From $gk=1+u$ and $kl=1+u^3$, one obtains $k=e^{-1}(1+u)$ and $l=e(1-u+u^2)$ in $\rr$, so $l+ue=e+u^2e$ in $S$. By Lemma~\ref{l:unit-action}, multiplication by the units $e,u$, and $u^2$ preserves additive atoms, so $e,ue,u^2e$ are distinct additive atoms. Hence Lemma~\ref{l:length}(2) forces $\mset{l,ue}=\mset{e,u^2e}$. If $l=e$, then $ue=u^2e$, and if $l=u^2e$, then $ue=e$, in either case cancellation in $\rr$ gives a positive power of $u$ equal to $1$, hence $u=1$, a contradiction. Assume next that $k=e$ is a unit. From $gk=1+u$ and $gh=1+u+u^2$, one obtains $h+uh=e+ue+u^2e$. If $M=\zplus(h)$, then multiplication by $u$ permutes additive atoms, so $M\uplus uM$ is an additive factorization of $h+uh$. The right side gives the additive factorization $\mset{e,ue,u^2e}$, and additive unique factorization forces $M\uplus uM=\mset{e,ue,u^2e}$, impossible because the two cardinalities are $2|M|$ and $3$. Finally, assume that $g$ and $k$ are nonunit additive atoms. From $gk=1+u$ and $kl=1+u^3$, we get $(1+u)l=g(1+u^3)$, that is, $l+ul=g+u^3g$. If $M=\zplus(l)$, then $M\uplus uM$ and $\mset{g,u^3g}$ are two additive factorizations of $l+ul$, so additive unique factorization gives $M\uplus uM=\mset{g,u^3g}$. Hence $M=\mset{m}$ for some additive atom $m$, and $\mset{m,um}=\mset{g,u^3g}$. If $m=g$, then $um=u^3g$, and so $u^2=1$. If $m=u^3g$, then $um=g$, and so $u^4=1$. Since $u>0$, either equality forces $u=1$, the desired contradiction.
\end{proof}
\smallskip

\begin{lemma} \label{l:singleton-packet}
	Let $S$ be a positive semiring whose additive monoid is a UFM. If $p,y,z \in S^\bullet$, the element $p$ is a nonunit, and $1+py=pz$, then $p$ is not prime in $S^\bullet$.
\end{lemma}

\begin{proof}
	Let $G=\gp(S,+)$, and identify $S$ with its image in $G$. We will construct $r_0 \in S^\bullet$ with $p \nmid r_0$ and then construct $Q_0 \in S^\bullet$ with $pQ_0=r_0^2$, contradicting primeness. By the additive-coordinate setup in Section~\ref{s:prelim}, $G$ is the free abelian group on the additive atoms, $S$ is the nonnegative coordinate cone, and the natural map $G \to \rr$ is an injective ring homomorphism. Hence multiplication by any nonzero element of $S$ is injective on $G$, because after the injective map $G \to \rr$ it is multiplication by a positive real number. The equality $1+py=pz$ gives $p(z-y)=1$ in $G$. Thus $z-y \neq 0$. If $z-y$ belonged to $S$, then it would be a multiplicative inverse of $p$, contrary to the assumption that $p$ is a nonunit. Hence $z-y \notin S$. Since $S$ is exactly the nonnegative coordinate cone in $G$, some additive coordinate of $z-y$ is negative.

	Write $y_i$ and $z_i$ for the coordinates of $y$ and $z$ in the additive-atom basis of $G$. The set of coordinates with $y_i>z_i$ is nonempty and finite, and for each such coordinate set $\rho_i=(y_i-z_i)/y_i$. Let $\rho=\max \rho_i$. Then $0<\rho\le 1$, and $\rho \in \qq$. The interval $[\rho/2,\rho)$ has positive length and lies in $(0,1]$, so it contains a positive rational number smaller than $1$. Choose positive integers $K<N$ such that $\rho/2 \le K/N<\rho$. Put $r_0=N+pKy \in S^\bullet$ and $q_0=N(z-y)+Ky=Nz-(N-K)y \in G$. Then $pq_0=r_0$. At a coordinate where $\rho_i=\rho$, the coefficient of $q_0$ is $N(z_i-y_i)+Ky_i=y_i(K-N\rho)<0$, so $q_0 \notin S$. If $p \mid r_0$, say $r_0=ps$ with $s \in S$, then $p(q_0-s)=0$ in $G$, and the injectivity of multiplication by $p$ gives $q_0=s$, contradicting $q_0 \notin S$. Thus $p \nmid r_0$.

	Now set $Q_0=N^2(z-y)+2NKy+pK^2y^2=N^2z+(2NK-N^2)y+pK^2y^2 \in G$. Then $pQ_0=r_0^2$. We claim that $Q_0 \in S$. The coordinate of $N^2(z-y)+2NKy$ at $i$ is $N^2(z_i-y_i)+2NKy_i$, it is nonnegative if $y_i \le z_i$, while if $y_i>z_i$ it equals $N^2y_i(2K/N-\rho_i)$, which is nonnegative because $\rho_i\le\rho\le 2K/N$. Hence $N^2(z-y)+2NKy$ belongs to $S$, and after adding $pK^2y^2 \in S$ we get $Q_0 \in S$. Since $pQ_0=r_0^2>0$, in fact $Q_0 \in S^\bullet$. Therefore $p \mid r_0^2$ but $p \nmid r_0$. If $p$ were prime, then $p \mid r_0$, a contradiction. Thus $p$ is not prime.
\end{proof}
\bigskip
\section{The Bi-UF Positive Conjecture}
\label{s:conjecture}

\smallskip

\begin{lemma} \label{l:terminal-atom}
	If $S$ is a bi-UFS positive semiring and $S \ne \nn_0$, then $S$ contains an irrational element belonging to $\mca_+(S) \cap \mca_\times(S)$.
\end{lemma}

\begin{proof}
	Since every UFM is an HFM, Proposition~\ref{p:rational-collapse} gives $S \cap \qq=\nn_0$. Take $s \in S \setminus \nn_0$. Then $s$ is irrational. In any additive factorization of $s$, at least one atom must be irrational: otherwise all atoms appearing would lie in $\nn_0$, and hence so would $s$. Also $1 \in \mca_+(S)$ by \cite[Proposition~3.1]{BCG21}, while every integer $n \ge 2$ decomposes as $1+\cdots+1$. Hence the only rational additive atom is $1$. Thus $S$ has an irrational additive atom $b$. By Theorem~\ref{t:no-units}, the only multiplicative unit is $1$, so $b$ is a multiplicative nonunit. Since $S^\bullet$ is a UFM, it is atomic. Choose a multiplicative atom $a$ dividing $b$ in $S^\bullet$. Lemma~\ref{l:divisor-closed} gives $a \in \mca_+(S)$, while $a \in \mca_\times(S)$ by construction. Since the only rational additive atom is $1$ and $1$ is not a multiplicative atom, the element $a$ is irrational.
\end{proof}

\smallskip

\begin{theorem} \label{t:cyclotomic}
	No bi-UFS positive semiring contains an irrational element belonging to $\mca_+(S) \cap \mca_\times(S)$.
\end{theorem}

\begin{proof}
	Suppose for the sake of contradiction that $a$ is irrational and belongs to $\mca_+(S) \cap \mca_\times(S)$. By \cite[Proposition~3.1]{BCG21}, $1 \in \mca_+(S)$. By Theorem~\ref{t:no-units}, $S^\bullet$ is reduced, and the standard factorial-monoid fact recalled in Section~\ref{s:prelim} shows that the atom $a$ is prime in $S^\bullet$. Set $P=1+a+a^2$, $Q=1+a^3$, $R=1+a$, and $T=1+a^2+a^4$. Since $PQ=RT$, Lemma~\ref{l:refinement} gives $u,v,w,x \in S^\bullet$ such that $P=uv$, $Q=wx$, $R=uw$, and $T=vx$. As $1$ and $a$ are distinct additive atoms, $\lplus(R)=2$, and Lemma~\ref{l:length}(1) gives $\lplus(u)\lplus(w) \le 2$. Thus the possible pairs $(\lplus(u),\lplus(w))$ are $(1,2)$, $(2,1)$, and $(1,1)$.

\smallskip
\indent \textsc{Case 1:} $\lplus(u)=1$ and $\lplus(w)=2$. Write $w=s+t$ with $s,t \in \mca_+(S)$. Since $uw=1+a$, Lemma~\ref{l:length}(2) applied to $us+ut=1+a$ shows that $\mset{us,ut}=\mset{1,a}$. Hence one of $us,ut$ is $1$, and so $u$ is a unit. Theorem~\ref{t:no-units} gives $u=1$. Thus $w=1+a$, and from $Q=wx$ we get $x=(1+a^3)/(1+a)=a^2-a+1 \in S$. Therefore $x+a=1+a^2$. Let $A=\zplus(x)$ and $B=\zplus(a^2)$. This equality gives $A\uplus\mset{a}=\mset{1}\uplus B$. Comparing coordinates, the $1$-coordinate shows that $A$ has one more copy of $1$ than $B$, the $a$-coordinate shows that $B$ has one more copy of $a$ than $A$, and every other coordinate agrees. Thus, after subtracting the common coordinatewise minimum, there is a common multiset $Y$ such that $B=Y\uplus\mset{a}$ and $A=Y\uplus\mset{1}$. Let $y \in S$ be represented by $Y$. Then $a^2=a+y$ and $x=1+y$. Since $y=0$ would imply $a=1$, we have $y \in S^\bullet$, and in $\rr$ we have $y=a^2-a=a(a-1)>0$. Thus $a>1$.

	Using $y=a^2-a$, we get $a+y+y^2=a^2(1+(a-1)^2)$ and $(1+(a-1)^2)(1+a)=2+ay$, so $(a+y+y^2)(1+a)=a^2(2+ay)$. All divisibility in the next sentence is in $S^\bullet$. Since $a$ is prime, either $a \mid (1+a)$ or $a^2 \mid (a+y+y^2)$: indeed, if $a \nmid (1+a)$, then primeness gives $a \mid (a+y+y^2)$, writing $a+y+y^2=aC$ and canceling one $a$ gives $C(1+a)=a(2+ay)$, so primeness and $a \nmid (1+a)$ give $a \mid C$. We rule out these two alternatives.

	First assume that $a \mid (1+a)$, say $ab=1+a$. From $\lplus(1+a)=2$ and Lemma~\ref{l:length}(1), we have $\lplus(b) \le 2$. If $\lplus(b)=2$, then after writing $b=c+d$ with $c,d \in \mca_+(S)$, Lemma~\ref{l:length}(2) applied to $ac+ad=1+a$ forces one of $ac,ad$ to be $1$, making $a$ a unit. Hence $b$ is an additive atom. Also $b=(1+a)/a=1+1/a>1$. The equalities $y=a^2-a$ and $ab=1+a$ give $1+by=a+y$. Let $M=\zplus(y)$. As $y \ne 0$, the multiset $M$ is nonempty. Expanding $1+by=a+y$ and applying Lemma~\ref{l:length}(2) shows in particular that each $bm$ with $m \in M$ is an additive atom, and gives $\mset{1}\uplus bM=\mset{a}\uplus M$. Lemma~\ref{l:scaling-chain} yields $M=\mset{1,b,\dots,b^{N-1}}$ and $a=b^N$ for some $N \ge 1$. If $N \ge 2$, then $a=b b^{N-1}$ is a proper factorization of the multiplicative atom $a$. Both $b$ and $b^{N-1}$ belong to $S^\bullet$ and are nonunits, since they are greater than $1$ and Theorem~\ref{t:no-units} leaves only $1$ as a unit. Thus $N=1$, and so $a=b$ and $a^2=a+1$. Then $(2+a)^2=5a^2$, so $a \mid (2+a)^2$ and, by primality, $a \mid (2+a)$. But $(2+a)/a=2a-1$ because $a(a-1)=1$, and Lemma~\ref{l:order-hole} excludes $2a-1 \in S^\bullet$. This contradiction rules out the first alternative.

	Now assume that $a^2 \mid (a+y+y^2)$, say $a^2h=a+y+y^2$ with $h \in S^\bullet$. Let $n \in \nn_0$ be the $a$-coordinate of $\zplus(a^2)$, and write $a^2=na+r$, where $r \in S$ has no $a$-coordinate. The equality $a^2=a+y$ gives $n \ge 1$, while $r \ne 0$ because otherwise $a$ would be rational. Hence $r \in S^\bullet$, and in $\rr$ we have $r=a(a-n)>0$, so $a>n$. Also $a \nmid r$: indeed, if $r=as$ with $s \in S$, then $s=a-n \in S^\bullet$ because $a>n$, contradicting Lemma~\ref{l:order-hole}. Dividing $a^2h=a+y+y^2$ by $a^2$ in $\rr$ gives $h=a^2-2a+2$, and so $h+a=y+2$. Since $y=a^2-a=(n-1)a+r$, comparison of the $a$-coordinate in $h+a=y+2$ gives $\operatorname{coeff}_a(h)+1=n-1$, so $n \ge 2$. Choose positive integers $c,d$ with $c+d=n$, and set $U=r^2+car$, $V=r^2+dar$, and $H=cd\,r^2+r^3$, all of which belong to $S^\bullet$. Then, using $r+na=a^2$, one obtains $UV=r^2(r+ca)(r+da)=a^2H$. Thus $a \mid UV$. We claim that $a \nmid U$ and $a \nmid V$. If $a \mid U=r(r+ca)$, then primeness gives $a \mid r$ or $a \mid (r+ca)$. The first has been excluded, and in the second case the quotient is $(r+ca)/a=a-d$, which is positive because $a>n=c+d>d$, and is forbidden by Lemma~\ref{l:order-hole}. The proof that $a \nmid V$ is the same, with quotient $(r+da)/a=a-c$, positive because $a>n>c$. This contradicts the primeness of $a$, and Case~1 is impossible.

\smallskip
\indent \textsc{Case 2:} $\lplus(u)=2$ and $\lplus(w)=1$. Write $u=s+t$ with $s,t \in \mca_+(S)$. Since $uw=1+a$, Lemma~\ref{l:length}(2) applied to $sw+tw=1+a$ shows that one of $sw,tw$ is $1$. Hence $w$ is a unit, and Theorem~\ref{t:no-units} gives $w=1$. Thus $u=1+a$, and from $P=uv$ we get $v=(1+a+a^2)/(1+a) \in S$. The equality $P=uv$ becomes $v+av=1+a+a^2=1+a(1+a)$. The element $1+a$ is a nonunit because it is greater than $1$ and Theorem~\ref{t:no-units} says that $1$ is the only unit. Choose a multiplicative atom $p$ dividing $1+a$, say $1+a=pc$ with $c \in S^\bullet$. Then $1+p(ac)=p(cv)$, where $cv \in S^\bullet$ because $c,v \in S^\bullet$. Applying Lemma~\ref{l:singleton-packet} with $(p,y,z)=(p,ac,cv)$ shows that $p$ is not prime, contradicting the primeness of atoms in the UFM $S^\bullet$. Thus Case~2 is impossible.

\smallskip
\indent \textsc{Case 3:} $\lplus(u)=\lplus(w)=1$. Then $u$ and $w$ are additive atoms. Neither is $1$, since otherwise the equality $uw=1+a$ would force the other to have additive length $2$. By Theorem~\ref{t:no-units}, both $u$ and $w$ are multiplicative nonunits. Choose a multiplicative atom $p$ dividing $u$, say $u=pc$ with $c \in S^\bullet$. From $uw=1+a$ and $uv=1+a+a^2$, with $v,w \in S^\bullet$, we obtain $p(cw)=1+a$ and $p(cv)=1+a+a^2=1+a(1+a)=1+p(acw)$. Thus $1+p(acw)=p(cv)$. Applying Lemma~\ref{l:singleton-packet} with $(p,y,z)=(p,acw,cv)$ makes $p$ nonprime, contradicting the primeness of atoms in the UFM $S^\bullet$. Thus Case~3 is impossible.

	The three cases exhaust all possibilities for $(\lplus(u),\lplus(w))$, and each has led to a contradiction. Therefore no such irrational element $a$ exists.
\end{proof}

\smallskip

\begin{theorem} \label{t:main}
	A positive semiring $S$ is a bi-UFS if and only if $S=\nn_0$.
\end{theorem}

\begin{proof}
	If $S=\nn_0$, then $(S,+)$ is freely generated by $1$, while $S^\bullet=\nn$ is the free commutative monoid on the ordinary primes. Hence $\nn_0$ is a bi-UFS. Conversely, suppose that $S$ is a bi-UFS positive semiring. If $S \ne \nn_0$, then Lemma~\ref{l:terminal-atom} gives an irrational element in $\mca_+(S) \cap \mca_\times(S)$, contradicting Theorem~\ref{t:cyclotomic}. Hence $S=\nn_0$.
\end{proof}

\bigskip

\end{document}